\documentclass{amsart}

\usepackage{amsthm}
\usepackage{amssymb}
\usepackage{amsmath}
\newtheorem{theorem}{Theorem}
\newtheorem{lemma}[theorem]{Lemma}

\theoremstyle{remark}

\numberwithin{theorem}{section} \numberwithin{equation}{section}

\begin{document}
\title{\textbf {Small prime gaps in abelian number fields}}
\author[Alexandra Musat]{ Alexandra Mihaela Musat}
\date{}
\maketitle
\begin{abstract}
In \cite{GPY}, Dan Goldston, Janos Pintz, and Cem Yildirim proved that there are infinitely many rational primes that are much closer together than the average gap between primes. 
Assuming the Elliott-Halberstam conjecture, they also proved the existence of infinitely many bounded differences between consecutive primes. 

We will prove an analogue of their result for a number field. We let $K$ be a finite extension of $\mathbb{Q}$ and consider $(q_n)$ the sequence of rational prime numbers that split completely in $K$. For $K$ an abelian extension of $\mathbb{Q}$, we prove that
\begin{equation*}
\displaystyle \liminf_{n \to \infty} \displaystyle \frac{q_{n+1}-q_n}{\log q_n}=0.
\label{teorema2}
\end{equation*}
Under the Elliott-Halberstam conjecture, we also prove that there are infinitely many bounded gaps between consecutive members of the sequence $q_n$.

We also give another proof of the same result in the special case of a quadratic extension of class number $1$, which relies on a generalization of the Bombieri-Vinogradov theorem for quadratic number fields. 

\end{abstract}
\section{Introduction}
One of the most famous unresolved problems in number theory is the twin prime conjecture, which states that there are infinitely many prime numbers $p$ for which $p+2$ is also prime. In 2005, putting to rest a long-standing open problem, Dan Goldston, Janos Pintz and Cem Yildirim proved that there are infinitely many primes that are very close together-much closer together than the average gap between primes. Under the Elliott-Halberstam conjecture, which concerns primes in arithmetic progressions, they also proved that there are infinitely many bounded gaps between consecutive prime numbers. More formally, they proved that
\begin{equation}
\liminf_{n \to \infty} (p_{n+1}-p_n) \le 16,
\label{ec1}
\end{equation} where $p_n$ denotes the $n^{\text{th}}$ prime number.

If we let 
\begin{equation*}
E= \displaystyle \liminf_{n\to\infty} \frac{p_{n+1}-p_n}{\log{p_n}},
\end{equation*}
then the main result proven by Goldston, Pintz and Yildirim in \cite{GPY} is that \begin{equation}
E=0,
\label{prim}
\end{equation} which is implied by the twin prime conjecture and by \eqref{ec1}.
 
 Finding upper bounds on $E$ has been a long studied problem. The famous prime number theorem states that the number $\pi(x)$ of prime numbers up to $x$ is roughly $\displaystyle \frac{x}{\ln x}.$ This is equivalent to the statement that $p_n \sim n \ln n$, which would imply that the average gap between consecutive prime numbers is $\ln n \sim \ln p_n$. This asymptotic expression implies that $E \le 1$. Erd\"{o}s \cite{erdos} was the first to prove that $E<1$. Another important result was that $E<\displaystyle \frac{1}{2}$, proven by Bombieri and Davenport in \cite{bombieri}. Maier \cite{maier} further improved the result to $\displaystyle E \le 0.24...$.
 
 In May 2005, Goldston, Pintz and Yildirim showed that $E=0$. Their proof uses the Bombieri-Vinogradov theorem on primes in arithmetic progressions and shows a connection between the distribution of primes in arithmetic progressions and small gaps between primes.  
 
 In the present paper, we will prove a generalization of this result for the case of an abelian number field. Following \cite{GPY}, we will prove both an unconditional result and a conditional result, depending on the Elliott-Halberstam conjecture. 
 We will prove the following unconditional result.
 
 \begin{theorem}[Unconditional result]
Let $K$ be an abelian extension of $\mathbb{Q}$ and let $(q_n)$ be the sequence of rational prime numbers that split completely in $K$. Then 

$$\displaystyle \liminf_{n \to \infty} \displaystyle \frac{q_{n+1}-q_n}{\log q_n}=0.$$

\label{goal}
\end{theorem}

 Assuming the Eliott-Halberstam conjecture, we will also prove the following theorem.
 
 \begin{theorem}[Conditional Result]
Let $K$ be an abelian extension of $\mathbb{Q}$ and let $(q_n)$ be the sequence of rational prime numbers that split completely in $K$. Then there exist infinitely many bounded gaps between consecutive members of the sequence $q_n$.
\label{goal2}
\end{theorem}
 Our method of proof of the two theorems above will follow the main ideas used in \cite{GPY}. Using results from class field theory, we know that the primes that split completely in a number field are described, with finitely many exceptions, by congruence conditions. In our proof, we will consider the analogue multiplicative functions as in \cite{GPY} and we will impose some extra congruence conditions on them.
 
 The paper is organized as follows. In section 2, we give a quick overview of the method used by Goldston, Pintz and Yildirim in the original proof in \cite{GPY}. In section 3, we will find estimates for the main sums used in proving the existence of conditionally bounded gaps and the unconditional result. In sections 4 and 5 we prove the conditional result \eqref{goal2} and the unconditional result \eqref{goal}, respectively. Section 6 of the paper provides a different approach for proving \eqref{goal} and \eqref{goal2} in the special case of a quadratic extension of class number $1$, which could give some insight into how we could prove analogous results for more general number fields.
 
 \section{A brief review of the Goldston, Pintz and Yildirim method}
 The general approach employed by Goldston, Pintz and Yildirim in \cite{GPY} was the following. 

 Let $k$ be a given positive integer which is at least 2 and take $H=\{h_1, \ldots, h_k\}$ to be a set of nonnegative distinct integers. Define

$$ \varpi(n)=
\begin{cases}
\log{p} & \text{ if }  n=p  \text{ prime} \\
0 & \text{ otherwise.}
\end{cases} $$ Let $P(n,H)=(n+h_1) \cdots (n+h_k) $ and suppose that for each prime $p$ the number $\nu_p$ of residue classes occupied by the elements of $H$ is at most $p.$ The Hardy-Littlewood conjecture implies that there exist infinitely many $n$ such that the translate set $n+H$ consists entirely of primes. 

The idea of Goldston, Pintz and Yildirim was to find a nonnegative function $\alpha(n)$ such that 
\begin{equation}
\displaystyle\sum_{x \le n<2x} \alpha(n) \displaystyle \sum_{h \in H} \varpi(n+h)/\displaystyle \sum_{x \le n<2x} \alpha(n)
\label{cevrem}
\end{equation}
 is big for $ x\to \infty. $ This would be useful for the following reason. If we could choose a function $\alpha(n)$ such that for $x \gg 1$,
 \begin{equation*}
 \displaystyle\sum_{x \le n<2x} \alpha(n) \displaystyle \sum_{h \in H} \varpi(n+h)/\displaystyle \sum_{x \le n<2x} \alpha(n) \ge (1+\epsilon) \ln x,
 \end{equation*} then it would follow that there would be at least two primes in the set $n+H$, for some $n \in [x, 2x)$, the gap being bounded by $h_k-h_1$. Letting $x \to \infty$ would prove the existence of bounded differences between consecutive primes.  
 Motivated by the ideas behind the Selberg sieve, they put $ \alpha(n,H)= (\displaystyle\sum_{d|P(n,H)} \lambda_d)^2 $, with $\lambda_d$ given by the formula
$$ \lambda_d=
\begin{cases}
\frac{1}{(k+l)!} \mu(d)(\log{\frac{R}{d}})^{k+l} & \text{ if } d \le R \\
0 & \text{ otherwise.}
\end{cases} $$
The best choice for the parameter $l$ above is positive but $o(k)$. This choice will become clearer later in the paper. 

In order to estimate the ratio \eqref{cevrem}, they invoked the Bombieri-Vinogradov theorem, which concerns the distribution of primes in arithmetic progressions. To state the Bombieri-Vinogradov theorem, we must introduce some notation. Define the classical von Mangoldt function 
\begin{equation*}
\Lambda(n)=
\begin{cases}
\log{p} & n=p^a\\
0 & \text{otherwise}
\end{cases}
\end{equation*} and the normalized prime-counting function 
\begin{equation*}
\psi(x;q,a)=\displaystyle\sum_{\substack{n \le x \\ n \equiv a \pmod {q}}} \Lambda(n)
\end{equation*} for the arithmetic progression $a \bmod {q}.$ Define the error quantity 
\begin{equation*}
E(x,q)=\displaystyle\max_{(a,q)=1} \displaystyle\max_{y \le x} |\psi(y;q,a)-\frac{y}{\phi(q)}|. 
\end{equation*} Then the Bombieri-Vinogradov theorem asserts that the error term above is as small on average over $q$ as the Riemann hypothesis would predict.
\begin{theorem}[Bombieri-Vinogradov]
For a real number $A$, there exists a $B$ such that 
\begin{equation*}
\sum_{q<\frac{x^{1/2}}{(\log x)^B}} E(x,q)\ll \frac{x}{(\log x)^A}.
\end{equation*}
\end{theorem} 
Let $\theta$ be the supremum of all $\theta'$ such that \begin{equation*} 
\sum_{q<x^{\theta'}} E(x,q)\ll_{\theta', A} \frac{x}{(\log x)^A}.
\end{equation*}
We call $\theta$ the \emph{level of distribution} of primes in arithmetical progressions. Note that the Bombieri-Vinogradov theorem implies that $\theta \ge \frac{1}{2}$. The Elliot-Halberstam conjecture claims that $\theta=1$. In \cite{GPY}, it is shown that the truth of the Elliott-Halberstam conjecture implies that $ \displaystyle\liminf_{n \to \infty} (p_{n+1}-p_n) \le 16$. 

\section{Some estimates needed for an abelian extension}
In this section, we will prove the main lemma we will use in the proof of \eqref{goal} and \eqref{goal2}.
\subsection{Main ideas used in the proof} 

It is known that the primes that split completely in an abelian extension are described, with finitely many exceptions, by congruence conditions. Therefore, we need to prove that the sequence of rational prime numbers $(q_n)$ with $q_n \equiv a \pmod m$, for fixed $a, m$ with $(a,m)=1$ satisfies \eqref{goal}. 

As in \cite{GPY}, we take $H$ to be a a set of $k$ positive integers $H=\{h_1, h_2, \ldots, h_k\}$ and we impose the extra condition that $H \bmod m =\{0\}$. Let 
\begin{equation}
m=p_1^{a_1} \cdots p_r^{a_r}
\label{factorizare}
\end{equation} be the prime factorization of $m$.
Choose 
\begin{equation}
 \alpha(n,H)= (\displaystyle\sum_{d|P(n,H)} \lambda_d)^2, 
 \label{alfa}
 \end{equation} with $\lambda_d$ given by the formula
\begin{equation}
 \lambda_d=
\begin{cases}
\frac{1}{(k+l)!} \mu(d)(\log{\frac{R}{d}})^{k+l} & \text{ if } d \le R \\
0 & \text{ otherwise,}
\end{cases} 
\label{lambda}
\end{equation} for some $R=R(x)$.

We will estimate $\displaystyle \sum_{\substack{x \le n<2x \\ n \equiv a\pmod m}} \alpha(n)$ and $\displaystyle \sum_{\substack{x \le n<2x \\ n\equiv a \pmod m}} \alpha(n) \sum_{h \in H} \varpi(n+h)$, and assuming the Elliott-Halberstam conjecture, we will prove that 
\begin{equation*}
\displaystyle \sum _{\substack{x\le n < 2x \\ n \equiv a \pmod m}} \alpha(n) \displaystyle \sum_{h \in H} \varpi(n+h) / \displaystyle \sum_{\substack{x \le n <2x\\ n \equiv a \pmod m}} \alpha(n) \ge (1+ \epsilon) \log x.
\end{equation*}

\subsection{Estimates}

For a multiplicative function $\rho$ on the squarefree integers, define 
$$ \mathfrak{S}(\rho)= \prod_{p \text{ prime }} \left(1-\frac{\rho(p)}{p}\right){\left(1-\frac{1}{p}\right)}^{-k}.$$ 

We impose the congruence condition $H \bmod m = \{0\}$ and we evaluate $\displaystyle\sum_{\substack{x \le n<2x \\ n\equiv a \pmod {m}}} \alpha(n)$ and $\displaystyle\sum_{\substack{x \le n<2x\\ n\equiv a \pmod {m}}} \alpha(n) \varpi(n+h_0)$, for $h_0$ not belonging to $H$ and $h_0 \equiv 0 \pmod m$, and for $h_0 \in H$.

The estimates are given by the following lemma.
\begin{lemma} \label{lema1}
Let $H$ be a set of $k$ nonnegative integers such that $H \bmod m=\{0\}$ and $$\max _{h \in H} h \ll \log R.$$ 

Let $\epsilon>0$ and $\alpha(n)$ and $\lambda_d$ be as given by formula \eqref{lambda}. Then for $R, x \to \infty$ such that $\log x \ll \log R$, we have:
\begin{enumerate}
\item  Provided $R \le x^{1/2-\epsilon},$ 
\begin{equation}
\displaystyle\sum_{\substack{x \le n<2x \\ n\equiv a \pmod {m}}} \alpha(n)= \frac{x}{m} \frac{(\log R)^{k+2l}}{(k+2l)!} \binom{2l}{l} (\mathfrak{S}(\rho_1)+o(1)), 
\label{estimate 1}
\end{equation} where the multiplicative function $\rho_1(p)$ is defined by: 
\begin{equation*}
\rho_1(p)= 
\begin{cases} 
0 & \text{ if } p=p_i, \text{ for } i \in \{1, \ldots,r\}  \\
\nu_p & \text{ if } p \text{ prime} \ne p_i, \text{ for any } i \in \{1, \ldots,r\} 
\end{cases}
\end{equation*} and extended multiplicatively to squarefree integers. Recall that $p_i$ are the primes that appear in the factorization of $m$, as defined in \eqref{factorizare}. \\
\item For an integer $h_0 \notin H$ such that $h_0 \equiv 0 \pmod m$, provided $ R \le x^{\theta/2-\epsilon}:$
\begin{equation}
\displaystyle \sum_{\substack{x \le n<2x\\ n\equiv a \pmod {m}}} \alpha(n) \varpi(n+h_0)= \frac{x}{\phi(m)} \frac{(\log R)^{k+2l}}{(k+2l)!} \binom{2l}{l} (\mathfrak{S}(\rho_2)+o(1)), 
\label{estimate 2}
\end{equation} where the multiplicative function $\rho_2(p)$ is defined by  
\begin{equation*} \rho_2(p)=
\begin{cases} 
0 & \text{ if } p=p_i, \text{ for } i \in \{1, \ldots,r\}\\
\displaystyle \frac{p ( \nu_p (H \cup \{h_0\})-1)}{\phi(p)} & \text{ if } p \text{ prime } \ne p_i, \text{ for any } i \in \{1, \ldots,r\}.
\end{cases}
\end{equation*}  Recall that $ \theta $ is the level of distribution of primes in arithmetic progressions, as defined in the previous section.\\
\item For an integer $h_0 \in H$, provided $ R \le x^{\theta/2-\epsilon},$
\begin{equation}
\displaystyle\sum_{\substack{x \le n<2x\\ n\equiv a \pmod {m}}} \alpha(n) \varpi(n+h_0)= \frac{x}{\phi(m)} \frac{(\log R)^{(k+2l+1)}}{(k+2l+1)!} \binom{2(l+1)}{l+1} (\mathfrak{S}(\rho_3)+o(1)), 
\label{estimate 3}
\end{equation} where the multiplicative function $\rho_3(p)$ is defined by  
\begin{equation*}
 \rho_3(p)=
\begin{cases} 
0 & \text{ if } p=p_i, \text{ for }i \in \{1, \ldots,r\} \\
 \displaystyle \frac{p (\nu_p(H)-1)}{\phi(p)} & \text{ if } p \text{ prime } \ne p_i, \text{ for any } i \in \{1, \ldots,r\}.
\end{cases}
\end{equation*}
\end{enumerate} 
\end{lemma}
\begin{proof}
We will first prove that the main term of the sum \eqref{estimate 1} is of the form 
\begin{equation*}
\displaystyle \frac{x}{m} \sum_{d_1,d_2 \le R} \lambda_{d_1} \lambda_{d_2} \frac{\rho_1([d_1,d_2])}{[d_1,d_2]}, 
\end{equation*} and that the main terms of the sums \eqref{estimate 2} and \eqref{estimate 3} are of the form $$\displaystyle \frac{x}{\phi(m)} \sum_{d_1,d_2 \le R} \lambda_{d_1} \lambda_{d_2} \frac{\rho([d_1,d_2])}{[d_1,d_2]},$$ where $\rho=\rho_2$ for the second sum and $\rho=\rho_3$ for the third sum. 
We first evaluate the following.
 \begin{align*}
 \displaystyle \sum_{\substack{x \le n<2x \\ n \equiv a\pmod m}} \alpha(n) & =\displaystyle \sum_{\substack{x \le n<2x \\ n \equiv a\pmod m}} {\left(\sum_{d|P(n,H))} \lambda_d \right)}^2
 \\
 & = \sum_{d_1,d_2 \le R} \lambda_{d_1} \lambda_{d_2} \displaystyle \sum_{\substack{x \le n<2x \\ n \equiv a \pmod m \\ \delta=[d_1,d_2]|P(n,H)}} 1. \end{align*}

If $(m, \delta) \ne 1$, then there is some prime divisor $p_i $ of $m$  $( i \in \{1, \ldots, r \})$ that also divides $\delta.$ If $ \delta|P(n,H)$, then there would exist some element $h \in H$ for which $n+h \equiv 0\pmod {p_i}$. Since $H$ was chosen so that $H \bmod m=\{0\}$, it follows that $n\equiv 0\pmod {p_i}$, which contradicts the condition $n \equiv a\pmod m, (a,m)=1. $ Hence we have the following.
$$\sum_{d_1,d_2 \le R} \lambda_{d_1} \lambda_{d_2} \displaystyle \sum_{\substack{x \le n<2x \\ n \equiv a \pmod m \\ \delta=[d_1,d_2]|P(n,H)}} 1= \sum_{d_1,d_2 \le R} \lambda_{d_1} \lambda_{d_2} \displaystyle \sum_{\substack{x \le n<2x \\ n \equiv a \pmod m \\ \delta|P(n,H)\\ (m,\delta)=1}} 1 .$$
Since $\delta$ is sqarefree, the condition $\delta|P(n,H)$ is equivalent to that for each prime number $p|\delta, $ there is an element $h \in H$ such that $n+h \equiv 0 \pmod p$. Since the elements of $H$ occupy $\nu_p$ congruence classes modulo $p$ , it follows that the condition $\delta|P(n,H)$ is equivalent to saying that $n$ belongs to one of $\nu_p$ congruence classes modulo $p$. Therefore, the conditions $n \equiv a \pmod m, \delta|P(n,H),$ $(m,\delta)=1$ are equivalent to that $n$ belongs to one of $\rho_1(\delta)$ congruence classes modulo $[m,\delta]=m\delta, $ where $\rho_1$ is defined by:
\begin{equation*}
\rho_1(p)= 
\begin{cases} 
0 & \text{ if } p|m \\
\nu_p & \text{ if } p \nmid m. 
\end{cases}
\end{equation*}
Then $$\displaystyle \sum_{\substack{x \le n<2x \\ n \equiv a \pmod m \\ \delta|P(n,H)\\ (m,\delta)=1}} 1= \displaystyle x \frac{\rho_1(\delta)}{m \delta}+ r_{\delta},$$ with $|r_{\delta}| \le \rho_1(\delta).$ Hence we get 
\begin{align*}
\displaystyle \sum_{\substack{x \le n<2x \\ n \equiv a\pmod m}} \alpha(n) & = \sum_{d_1,d_2 \le R} \lambda_{d_1} \lambda_{d_2} \displaystyle \sum_{\substack{x \le n<2x \\ n \equiv a \pmod m \\ \delta|P(n,H)\\ (m,\delta)=1}} 1 
\\
& = \frac{x}{m} \displaystyle \sum_{d_1,d_2 \le R} \lambda_{d_1} \lambda_{d_2} \frac{\rho_1(\delta)}{\delta}+\displaystyle \sum_{d_1,d_2 \le R} \lambda_{d_1} \lambda_{d_2} r_{\delta}. 
\end{align*}
We will show that the error term 
\begin{equation}
\displaystyle \sum_{d_1,d_2 \le R} \lambda_{d_1} \lambda_{d_2} r_{\delta}
\label{eroare}
\end{equation}
is $o(x)$ for fixed $k$ and $l,$ provided that $R \le x^{1/2-\epsilon}$. 

We have the inequality 
\begin{equation*}
|r_{\delta}| \le \rho_1(\delta)= \prod_{p|\delta} \rho_1(p) \le \prod_{p|\delta} k= k^{\omega(d)},
\end{equation*} where $\omega(d)$ is the number of distinct prime divisors of $d.$
Using the definition of $\lambda_d$, we get that $\lambda_d \ll (\log R)^{k+l}$. Combining these two relations gives that the error term \eqref{eroare} is 
$$ \ll (\log R)^{2(k+l)} \sum_{\delta \le R^2} r_{\delta} \displaystyle \sum_{\substack{d_1,d_2 \le R \\ [d_1,d_2]=\delta}} 1.$$
Since $d_1, d_2$ are squarefree, we have that $\displaystyle \sum_{\substack{d_1,d_2 \le R \\ [d_1,d_2]=\delta}} 1=3^{\omega(d)},$ and hence the error term \eqref{eroare} is
 $ \ll (\log R)^{2(k+l)} \sum_{\delta \le R^2} (3k)^{\omega(d)}.$ Using the estimate $$\sum_{n<x} k^{\omega(n)} \ll x{(\log x)}^{k-1}, $$ we get that the error term is 
$ \ll R^{2+\epsilon'},$ for any $ \epsilon'>0.$ Provided $R \le x^{1/2-\epsilon},$ it follows that the error term is indeed $o(x).$\\

We now turn to evaluating the second sum. Let $h_0 \notin H, h_0 \equiv 0 \pmod m.$ Then
\begin{align*} 
 \sum_{\substack{x \le n<2x\\ n\equiv a \pmod {m}}} \alpha(n) \varpi(n+h_0) 
&= 
\displaystyle \sum_{\substack{x \le n<2x \\ n \equiv a \pmod m}} \varpi(n+h_0) \left(\sum_{d|P(n,H)} \lambda_d\right)^2 \\
&= \displaystyle \sum_{\substack{x \le n<2x \\ n \equiv a \pmod m}} \varpi(n+h_0) \sum_{\delta|P(n,H)} \lambda_{d_1} \lambda_{d_2}
\\
&= 
\sum_{d_1,d_2 \le R} \lambda_{d_1} \lambda_{d_2} \displaystyle \sum_{\substack{x \le n<2x\\ n \equiv a \pmod m \\ \delta|P(n,H)}} \varpi(n+h_0) \\
&= \sum_{d_1,d_2 \le R} \lambda_{d_1} \lambda_{d_2} \displaystyle \sum_{\substack{ 0 \le t< \delta\\ \delta|P(t,H)\\(\delta,m)=1}} \vartheta^{*}(x;m\delta,at+h_0) \\
&= \sum_{d_1,d_2 \le R} \lambda_{d_1} \lambda_{d_2} \displaystyle \sum_{\substack{0 \le t< \delta\\ \delta|P(t,H)\\(at+h_0,m\delta)=1}} \frac{x}{\phi(m\delta)} \\
&+ \sum_{d_1,d_2 \le R} \lambda_{d_1} \lambda_{d_2}\left( \displaystyle \sum_{\substack{ 0 \le t< \delta\\ \delta|P(t,H)\\(\delta,m)=1}} \vartheta^{*}(x;m\delta,at+h_0)-\frac{x}{\phi(m\delta)}\right).
\end{align*} where 
$$ \vartheta(x;q,a)= \displaystyle \sum_{\substack{ n \le x \\ n \equiv a \pmod m}} \varpi(n), $$ and
$$ \vartheta^{*} (x;q,a)= \vartheta(2x;q,a)-\vartheta(x;q,a). $$
As in the previous case, the conditions on $n$ imply that $(m, \delta) \ne 1.$ Therefore 
\begin{equation}
\displaystyle \sum_{\substack{x \le n<2x\\ n\equiv a \pmod {m}}} \alpha(n) \varpi(n+h_0)= \sum_{d_1,d_2 \le R} \lambda_{d_1} \lambda_{d_2} \displaystyle \sum_{\substack{x \le n<2x\\ n \equiv a \pmod m \\ \delta|P(n,H)\\(m,\delta)=1}} \varpi(n+h_0). 
\label{ec10}
\end{equation}
The conditions on $n$ in the equation \eqref{ec10} above are equivalent to that for any prime $p|\delta,$ there is some $h$ in $H \cup \{h_0\}$ such that $n = -h \bmod p,$ and $n \ne -h_0 \bmod p$. Hence $n$ belongs to one of $\nu_p (H \cup \{h_0\})-1$ congruence classes modulo $p.$ If we define the multiplicative function $\rho_2$ on the squarefree integers as
\begin{equation*}
 \rho_2(p)=
\begin{cases} 
0 & \text{ if } p|m\\
 \displaystyle \frac{p ( \nu_p (H \cup \{h_0\})-1)}{\phi(p)} & \text{ if } p \nmid m, 
\end{cases} 
\end{equation*}
and since for $(m,\delta)=1$ we have that $\phi(m\delta)=\phi(m)\phi(\delta)$, the main term of the second sum \eqref{estimate 2}  is $$\displaystyle \frac{x}{\phi(m)} \sum_{d_1,d_2 \le R} \lambda_{d_1} \lambda_{d_2} \frac{\rho_2([d_1,d_2])}{[d_1,d_2]},$$ and we will prove that the error term is $o(x).$

Recall our assumption that $\displaystyle \sum_{q<x^{\theta}} E(x,q)\ll_{\theta, A} \frac{x}{(\log x)^A}, $ for any $\epsilon>0$.

Since $\lambda_d \ll (\log R)^{k+l}$, the error term becomes 
$$ \ll (\log R)^{2(k+l)} \displaystyle \sum_{d_1,d_2 \le R} \displaystyle \sum_{\substack{ 0 \le t< \delta\\ \delta|P(t,H)\\(\delta,m)=1}} E(x;m\delta,at+h_0).$$
In the sum over $t, $ there are fewer than $k^{\omega(\delta)}$ terms and using a similar argument as in the evaluation of the error term for the first sum, we get that the error is
$$ \ll (\log R)^{2(k+l)} \sum_{\delta \le R^2} (3k)^{\omega(\delta)} \max_{(b,m\delta)=1} E(x; m\delta, b).$$
The Cauchy-Schwarz inequality shows that the above is 
$$ \ll (\log R)^{2(k+l)}{\left( \sum_{\delta \le R^2} \frac{(9k^2)^{\omega(\delta)}}{m\delta}\right)}^{1/2} {\left ( \sum_{\delta \le R^2} m\delta  \max_{(b,m\delta)=1} E(x; m\delta, b)^2 \right)}^{1/2}. $$
To estimate $\displaystyle {\left( \sum_{\delta \le R^2} \frac{(9k^2)^{\omega(\delta)}}{m\delta}\right)}^{1/2},$ we will use partial summation in the form of Abel's identity, which states the following.
\begin{lemma}[Abel's identity]
For any arithmetical function $a(n)$ let
\begin{equation*}
A(x)= \sum_{n\le x} a(n),
\end{equation*} where $A(x)=0$ if $x \le 1$. Assume $f$ has a continous derivative on the interval $[y,x]$, where $0<y<x$. Then we have
\begin{equation*}
\sum_{y < n \le x} a(n) f(n) = A(x)f(x)-A(y)f(y)- \int_{y}^{x} A(t) f'(t)\, dt
\end{equation*}
\end{lemma}
\begin{proof}
See \cite{apostol}.
\end{proof}
Now we apply the lemma above for $f(x)=\displaystyle \frac{1}{x}$, $a(n)= \displaystyle (9k^2)^{\omega(n)}$ and $y=1$, and we get that
\begin{equation}
\sum_{n\le x} \frac{(9k^2)^{\omega(n)}}{n}= \frac{ \displaystyle \sum_{n \le x} {(9k^2)^{\omega(n)}}} {x} + \int_{1}^{x} \frac{ \displaystyle \sum_{n \le t} (9k^2)^{\omega(n)}} {t^2} \, dt.
\label{abel2}
\end{equation}
Using the well-known asymptotics $\displaystyle \sum_{n\le x} k^{\omega(n)} \ll x \mathcal{L},$  where $\mathcal L$ is defined as $\mathcal L= O((\log N)^{O(1)}), $ and \eqref{abel2}, we get that 
$
\sum_{n\le x} \frac{(9k^2)^{\omega(n)}}{n} \ll \mathcal{L},$ hence 
\begin{equation}
{\left( \sum_{\delta \le R^2} \frac{(9k^2)^{\omega(\delta)}}{m\delta}\right)}^{1/2} \ll \mathcal{L}.
\label{schwarz}
\end{equation}
Now we will use the asymptotics 
\begin{equation}
 \sum_{q< R^2} q E(N,q)^2 \ll N \mathcal L \sum_{q<R^2} E(N,q),
 \label{asymptotics}
 \end{equation} which follows from the fact that 
 \begin{equation*}
 E(N,q) \ll \frac{N}{q} \log N ,
 \label{ec12}
 \end{equation*} for $q \ll N, $ which in turn is obtained after combining the following two estimates below:
 \begin{equation*}
 \frac{N}{\phi(q)} \ll \frac{N}{q} \log N
 \label{ec13}
 \end{equation*} 
 \begin{equation*}
 \psi(N;q,a)= \displaystyle \sum_{ p \equiv a \pmod q} \log p \le (\log N) (\frac{N}{q} + O(1)) \ll \frac{N}{q} \log N.
 \label{ec14}
 \end{equation*}
 Using the above asymptotics \eqref{asymptotics}, we get that
$$ \sum_{\delta \le R^2} m\delta  \max_{(b,m\delta)=1} E(x; m\delta, b)^2 \ll X \mathcal L \sum_{m\delta \le mR^2} E(x,m\delta).$$
For $R^2 \ll_{m} x^{\theta-\epsilon},$ we get that $\displaystyle \sum_{m\delta \le mR^2} E(x,m\delta) \ll \frac{x}{(\log x)^A}, $ so the error term will be
$$ \ll \mathcal L \left(x \mathcal L \frac{x}{(\log x)^{A}} \right)^{1/2} \ll \frac{x}{(\log x)^{A'}}.$$
Therefore the error term is $o(x)$ provided that $R \le x^{\theta/2-\epsilon}.$ \\

If $h_0 \in H, $ then we can replace $H$ by $H \setminus \{h_0\}$ and $(k,l)$ by $(k-1,l+1)$, since $\alpha(n,H)$ is thus preserved, and then we can apply the same arguments as for the case $h_o \in H$.

We have proved that the main terms are of the desired form. Let $$T=\sum_{d_1,d_2 \le R} \lambda_{d_1} \lambda_{d_2} \frac{\rho_([d_1,d_2])}{[d_1,d_2]},$$ for $\rho$ a multiplicative function on the squarefree integers. Using the estimate proven in \cite{GPY}, we get that 
$$T \sim \mathfrak{S}(\rho) \frac{(\log R)^{k+2l}}{(k+2l)!} \binom{2l}{l}.$$
Combining this result with the previous asymptotics gives the desired estimates. This finishes the proof of Lemma \eqref{lema1} which will be used in the next sections for proving \eqref{goal} and \eqref{goal2}. 

\end{proof}

\section{Conditionally bounded gaps}

The object of this section is to prove \eqref{goal2}, which is conditional on the Elliott-Halberstam conjecture.
We evaluate the ratio $$\displaystyle \sum_{\substack{x \le n<2x \\ n\equiv a \pmod m}} \alpha(n) \sum_{h \in H} \varpi(n+h) / \displaystyle \sum_{\substack{x \le n<2x \\ n\equiv a \pmod m}} \alpha(n), $$ for the choice of $\alpha$ made.

Using the lemma proved in the previous section, we get that for any $\delta>0$, the inequality 
$$\displaystyle \sum_{\substack{x \le n<2x \\ n\equiv a \pmod m}} \alpha(n) \sum_{h \in H} \varpi(n+h) / \displaystyle \sum_{\substack{x \le n<2x \\ n\equiv a \pmod m}} \alpha(n) \ge \left(\frac{m}{\phi(m)} \frac {\mathfrak{S}(\rho_3)}{\mathfrak{S}(\rho_1)} \beta -\delta \right) \log R $$ holds for sufficiently large $R,$ where 
$$ \beta= \frac{2k(2l+1)}{(l+1)(k+2l+1)}.$$
Now by definition  $$ \frac {\mathfrak{S}(\rho_3)}{\mathfrak{S}(\rho_1)}= \displaystyle \frac {\displaystyle \prod_{p \text{ prime }} \left(1-\frac{\rho_3(p)}{p}\right){\left(1-\frac{1}{p}\right)}^{-k+1}} {\displaystyle \prod_{p \text{ prime }} \left(1-\frac{\rho_1(p)}{p}\right){\left(1-\frac{1}{p}\right)}^{-k}}.$$
For any prime $p \nmid m$ we have  $$ 
\displaystyle \frac{\displaystyle \left(1-\frac{\rho_3(p)}{p}\right){\left(1-\frac{1}{p}\right)}^{-k+1}} {\displaystyle \left(1-\frac{\rho_1(p)}{p}\right){\left(1-\frac{1}{p}\right)}^{-k}}= \displaystyle \frac{ \left(1- \displaystyle \frac{\nu_p(H)-1}{p-1} \right)\displaystyle \left(1- \frac{1}{p} \right)} { \displaystyle 1- \frac{\nu_p(H)}{p} }=1.$$
On the other hand, for $p|m$ we have $\nu_p(H)=1$, so that 
$$\displaystyle \frac{\displaystyle \left(1-\frac{\rho_3(p_i)}{p_i}\right){\left(1-\frac{1}{p_i}\right)}^{-k+1}} {\displaystyle \left(1-\frac{\rho_1(p_i)}{p_i}\right){\left(1-\frac{1}{p_i}\right)}^{-k}}= \displaystyle 1-\displaystyle \frac{1}{p_i}.$$
Hence
 $$ \frac {\mathfrak{S}(\rho_3)}{\mathfrak{S}(\rho_1)}= \displaystyle \frac {\displaystyle \prod_{p \text{ prime }} \left(1-\frac{\rho_3(p)}{p}\right){\left(1-\frac{1}{p}\right)}^{-k+1}} {\displaystyle \prod_{p \text{ prime }} \left(1-\frac{\rho_1(p)}{p}\right){\left(1-\frac{1}{p}\right)}^{-k}}=\displaystyle \prod_{i=1}^ r \left(1-\frac{1}{p_i} \right)= \displaystyle \frac{\phi(m)}{m}. $$
 Thus, we get that
 $$\displaystyle \sum_{\substack{x \le n<2x \\ n\equiv a \pmod m}} \alpha(n) \sum_{h \in H} \varpi(n+h) / \displaystyle \sum_{\substack{x \le n<2x \\ n\equiv a \pmod m}} \alpha(n) \ge (\beta-\delta) \log R. $$  
 If we assume that $\theta>\frac{1}{2}$ and if we let $k,l \to \infty$ with $l=o(k)$, then $\beta \to 4.$ Thus for a sufficiently small $\delta,$ there exists some small $\epsilon>0$ such that $(\beta-\delta) \log R \ge (1+\epsilon)\log x, $ for $R,x$ large enough.

Combining the above two inequalities gives 
$$\displaystyle \sum_{\substack{x \le n<2x \\ n\equiv a \pmod m}} \alpha(n) \sum_{h \in H} \varpi(n+h) / \displaystyle \sum_{\substack{x \le n<2x \\ n\equiv a \pmod m}} \alpha(n) \ge( 1+\epsilon)\log x, $$
so any tuple with $k$ elements each congruent to $0 \bmod m$ would have infinitely many translates containing two primes congruent to $a \bmod m. $ Therefore,  
\begin{equation*}
\displaystyle\liminf_{n \to \infty}(q_{n+1}-q_n)< \infty.
\end{equation*} 
\section{The unconditional Result}
To obtain the unconditional result 
\begin{equation}
 \displaystyle \liminf_{n \to \infty} \frac{q_{n+1}-q_n}{\log q_n}=0, 
 \label{ec15}
\end{equation} we fix $\delta>0, $ and we define the function
$ g(n)= \displaystyle \sum_{\substack{ H \subseteq A \\ |H|=k}} \alpha(n,H), $ where $$A=[1,m\delta \log x] \cap m \mathbb{Z}.$$
We will prove that 
\begin{equation}
 \displaystyle \sum_{\substack{x \le n<2x \\ n\equiv 0 \pmod m}} g(n) \sum_{\substack{h\equiv a \pmod m \\ h<m\delta\log x}} \varpi(n+h)/ \displaystyle \sum_{\substack{x \le n<2x \\ n\equiv a \pmod m}} g(n)\ge (1+\epsilon) \log x  .
 \label{ec16}
 \end{equation}
 
Using the estimates proven before, we get that 
$$ \displaystyle \sum_{\substack{x \le n<2x \\ n\equiv 0 \pmod m}} g(n) \sim \frac{x}{m} \frac{(\log R)^{k+2l}}{(k+2l)!} \binom{2l}{l} \sum_{\substack{|H|=k\\ H\subseteq A}} \mathfrak{S}(\rho_1).$$
For $h \in H,$
 \begin{align*}
  \sum g(n) \sum_{h \in H} \varpi(n+h) &= \displaystyle \sum_{\substack{|H|=k\\ H\subseteq A}} \sum_{\substack{x \le n<2x \\ n\equiv a \pmod m}} \sum_{h \in H} \varpi(n+h)
 \\
 &\sim k \frac{x}{\phi(m)} \frac{(\log R)^{k+2l+1}}{(k+2l+1)!} \binom{2(l+1)}{l+1} \sum_{\substack{|H|=k\\ H\subseteq A}} \mathfrak{S}(\rho_3). 
 \end{align*}
 For $h \notin H,$
 \begin{align*}
   \sum g(n) \sum_{h \notin H} \varpi(n+h)&= \sum_{\substack{|H|=k+1 \\ H \subseteq A}} g(n) \sum_{h \in H} \varpi(n+h) 
 \\
  & \sim \frac{x}{\phi(m)} \frac{(\log R)^{k+2l}}{(k+2l)!} \binom{2l}{l} \sum_{\substack{|H|=k\\ H\subseteq A}} \mathfrak{S}(\rho_2). 
  \end{align*}
 Hence, \eqref{ec16} is equivalent to proving that
 \begin{equation}
  \beta \frac{x}{\phi(m)} \log R \sum_{\substack{|H|=k\\ H\subseteq A}} \mathfrak{S}(\rho_3) + \frac{1}{\phi(m)} \sum_{\substack{|H|=k\\ H\subseteq A}} \mathfrak{S}(\rho_2)   \ge (1+\epsilon) \log x \displaystyle \frac{1}{m} \sum_{\substack{|H|=k\\ H\subseteq A}} \mathfrak{S}(\rho_1).
  \label{ec17}
  \end{equation}
  Notice that $\mathfrak{S}(\rho_2)=\mathfrak{S}(H \cup \{h_0\}), $ and $\mathfrak{S}(\rho_3) / \mathfrak{S}(\rho_1)=\displaystyle \frac{\phi(m)}{m}.$
  
  We also have 
   $$\mathfrak{S}(\rho_1)=\mathfrak{S}(H)/ \displaystyle \prod_{i=1}^ r \left(1-\frac{1}{p_i} \right)= \mathfrak{S}(H) \displaystyle \frac{m}{\phi(m)}.$$ 
   Hence, $\mathfrak{S}(\rho_3)= \mathfrak{S}(H).$
  Then the desired inequality \eqref{ec17} is equivalent to
  \begin{equation*}
 \frac{\beta}{\phi(m)} \log R \sum_{\substack{|H|=k \\ H \subseteq A}} \mathfrak{S}(H) + \frac{1}{\phi(m)}\sum_{\substack{|H|=k+1 \\ H \subseteq A}}\mathfrak{S}(H)   \ge (1+\epsilon) \log X \frac{1}{\phi(m)} \sum_{\substack{|H|=k \\ H \subseteq A}} \mathfrak{S}(H),
 \end{equation*}
which is equivalent to 
  
  \begin{equation}
  \beta\log R \displaystyle \sum_{\substack{|H|=k \\ H \subseteq A}} \mathfrak{S}(H) + \displaystyle \sum_{\substack{|H|=k+1 \\ H \subseteq A}}\mathfrak{S}(H) \ge (1+\epsilon) \log x \displaystyle \sum_{\substack{|H|=k \\ H \subseteq A}} \mathfrak{S}(H).
 \label{ec18}
  \end{equation}

In order to establish \eqref{ec18}, we will now use the following result proven in \cite{GPY2}, section 14.
\begin{lemma}
Let 
$$ B_{A} (k)= \displaystyle \sum_{\substack{|H|=k \\ H \subseteq A}} \mathfrak{S}(H),$$
where all sets $H \subseteq A \subseteq [1,N]$ are counted with $k!$ multiplicity and $|A|=h.$  Let 
$$S_A^{*}(k)=\frac{B_A(k)}{h^k}.$$ If $k<\epsilon(h)h/\log_2 N, $ then
$$ S_A^{*}(k+1) \ge S_A^{*}(k)(1+O(\epsilon(h))+O(\frac{1}{\log N})).$$
\end{lemma}
We can apply the above lemma for $A=[1,m\delta \log x] \cap m \mathbb{Z}$, for which $|A|=\delta\log x+ O(1).$ By the lemma, we have that $B_A(k)/|A|^k$ is, apart from a factor of $1+o(1)$, non-decreasing as a function of $k.$ We divide \eqref{ec18} by $\sum_{\substack{|H|=k \\ H \subseteq A}} \mathfrak{S}(H).$ Thus, \eqref{ec18} is equivalent to proving that
\begin{equation*}
\beta \log R + \delta \log X \ge (1+\epsilon) \log x.
\end{equation*}
Taking $\beta \to 4$ and $R\to x^{\theta/2}=x^{1/4},$ notice that the above is true for any $\epsilon<\delta.$
Hence we get that 
$$\displaystyle \liminf_{n \to \infty} \frac{q_{n+1}-q_n}{\log q_n} \le \delta,$$ and since $\delta$ can be chosen to be arbitrarly small, we get the desired result.

\section{An alternate approach for quadratic fields}
 We will now give a second proof of \eqref{goal2} in the special case when $K$ is a quadratic extension of $\mathbb{Q}$ of class number $1$. The proof relies on a generalization of the Bombieri-Vinogradov theorem for the case of a quadratic number field, proven by E. Fogels in \cite{fogels}. 
 
 Although the result we prove in this section is a special case of \eqref{goal2}, the method we use could potentially be adapted to prove results about small prime gaps in more general number fields.
 
 We first give a brief overview of Fogels' result. In \cite{fogels}, Fogels gives a generalization of the Bombieri-Vinogradov theorem in the special case of a quadratic number field, in which he considers rational primes that split completely in the number field. 

Let $K$ be a fixed quadratic field. $\mathfrak{R}$ stands for classes of ideals. For any natural number $q$, we consider the group formed by the reduced classes of residues $ l \bmod q$ formed by the residues of the idealnorms $N(\mathfrak{a})$ with $(\mathfrak{a},[q])=1$ and $\mathfrak{a}$ belonging to the principal ideal class $\mathfrak{R}_1$. Denote by $\phi_1(q)$ the order of this group and by $h$ the number of classes of the ideal class group . Then Fogels' result is the following.
\begin{theorem}[Fogels]
Let $a=a(q, \mathfrak{R})$ be a normresidue$\pmod q$ with $(a,q)=1$ for the class $\mathfrak{R}$ of ideals in the quadratic number field $K$ and let $\pi(x; \mathfrak{R},q,a)$ denote the number of primes $p \le x$ such that $p \equiv a \pmod q$ and such that $p=N(\mathfrak{p})$ with $\mathfrak{p} \in \mathfrak{R}$. Then for any constant $A>0$ there is a corresponding constant $B>0$ such that
\begin{equation*}
\displaystyle \sum_{\mathfrak{R}} \displaystyle \sum_{q \le z^{1/2}/ (\log z)^{-B}} \max_{a(q, \mathfrak{R})} \max_{x \le z} |\pi(x; \mathfrak{R},q,a)-\frac{1}{h \phi_1(q)} \text{Li}(x)| \ll \frac{z}{(\log z)^A},
\end{equation*} with the constant in the notation depending merely on $A$ and the discriminant of the field.
\end{theorem}

To prove \eqref{goal} for the case of a quadratic number field, we proceed as follows.

\begin{proof}
Let $K$ be a fixed quadratic field of discriminant $D$ with class number $h=1$, and we take $H=\{h_1, \ldots, h_k\}$ to be a set of nonnegative distinct integers such that $H \bmod D =\{0\}$. We define the characteristic function $f$ on the integers by
\begin{equation*}
f(n)=
\begin{cases} 
1 & \text{ if } n \text{ is prime and splits completely in } K \\
0 & \text{ otherwise. }
\end{cases}
\end{equation*} 
Let $\alpha(n)$ be as defined in \eqref{alfa}. We will estimate $\displaystyle \sum_{x \le n < 2x} \alpha(n)$. We will also estimate  $\displaystyle \sum_{x\le n < 2x} \alpha(n) f(n+h_0)$, for $h_0$ an integer not belonging to $H$ such that $h_0 \equiv 0 \pmod D$  and for $h_0$ belonging to $H$, respectively.

Using the estimate in the original paper we get that 
\begin{equation}
\displaystyle \sum_{x \le n<2x} \alpha(n)= \ x \frac{(\log R)^{k+2l}}{(k+2l)!} \binom{2l}{l} (\mathfrak{S}(H)+o(1)).
\label{orig}
\end{equation}

Now let $h_0 \notin H$ be an integer such that $h_0 \equiv 0 \pmod D$. We will denote by $\delta=[d_1,d_2]$. We define $\pi^{*}(x;q,a)$ by
$ \pi^{*}(x;q,a)=\pi(2x;q,a)-\pi(x;q,a),$
 where $\pi(x;q,a)$ denotes the number of primes $p \le x$ such that $p \equiv a \pmod q$ and such that $p=N(\mathfrak{p})$ for $\mathfrak{p}$ an ideal in the number field (which is a principal ideal, since $K$ has class number one). Then we have 
\begin{align*} 
 \sum_{x \le n<2x} \alpha(n) f(n+h_0)&= \displaystyle \sum_{x \le n<2x } f(n+h_0) \left(\sum_{d|P(n,H)} \lambda_d\right)^2 
\\
&= \displaystyle \sum_{x \le n <2x} f(n+h_0) \sum_{\delta|P(n,H)} \lambda_{d_1} \lambda_{d_2}
\\
&= \sum_{d_1,d_2 \le R} \lambda_{d_1} \lambda_{d_2} \displaystyle \sum_{\substack{x \le n<2x \\ \delta|P(n,H)}} f(n+h_0)
\\
&= \sum_{d_1,d_2 \le R} \lambda_{d_1} \lambda_{d_2} \displaystyle \sum_{\substack{ m= n \bmod \delta\\ \delta|P(m,H)\\(\delta,m+h_0)=1}} \pi^{*}(X;\delta,m+h_0)
\\
&= \sum_{d_1,d_2 \le R} \lambda_{d_1} \lambda_{d_2} \displaystyle \sum_{\substack{m=n \bmod \delta\\ \delta|P(m,H)}} \frac{\text{Li}(x)}{\phi_1(\delta)} 
\\
&+ \sum_{d_1,d_2 \le R} \lambda_{d_1} \lambda_{d_2}\left( \displaystyle \sum_{\substack{m=n \bmod \delta\\ \delta|P(m,H)}} \pi^{*}(x;\delta,m+h_0)-\frac{\text{Li}(x)}{\phi_1(\delta)}\right).
\end{align*} 

Now we will evaluate the main term 
\begin{equation}
\sum_{\substack{d_1,d_2 \le R}} \lambda_{d_1} \lambda_{d_2} \displaystyle \sum_{\substack{m=n \bmod \delta\\ \delta|P(m,H)}} \frac{\text{Li}(x)}{\phi_1(\delta)}
\label{mainquadratic}
\end{equation} and the error term  
\begin{equation}
\sum_{d_1,d_2 \le R} \lambda_{d_1} \lambda_{d_2}\left( \displaystyle \sum_{\substack{m=n \bmod \delta\\ \delta|P(m,H)}} \pi^{*}(x;\delta,m+h_0)-\frac{\text{Li}(x)}{\phi_1(\delta)}\right).
\label{eroare2}
\end{equation} 
The conditions on $n$ in the main term \eqref{mainquadratic} above are equivalent to that $n$ belongs to one of $\nu_p(H \cup \{h_0\})-1$ congruence classes modulo $p$, by a similar argument as when evaluating \eqref{ec10}. We define the multiplicative function $\rho'$ by
\begin{equation*}
 \rho'(p)=
 \displaystyle \frac{p ( \nu_p (H \cup \{h_0\})-1)}{\phi_1(p)},  \text{ for } p \text{ prime }
\end{equation*} and extend it multiplicatively to squarefree integers. 
Since for $(m,\delta)=1$ we have that $\phi_1(m\delta)=\phi_1(m)\phi_1(\delta)$, the main term \eqref{mainquadratic} is \begin{equation}
\displaystyle \text{Li}(x) \sum_{d_1,d_2 \le R} \lambda_{d_1} \lambda_{d_2} \frac{\rho'([d_1,d_2])}{[d_1,d_2]} \sim \text{Li}(x) \frac{(\log R)^{k+2l}}{(k+2l)!} \binom{2l}{l} \mathfrak{S}(\rho'),
\label{mainterm3}
\end{equation} where the last asymptotics follows using the result proven in the original paper \cite{GPY}.  We have the following identity
\begin{equation*}
\mathfrak{S}(\rho')= \prod_{p \text{ prime }} \left(1-\frac{\rho'(p)}{p}\right){\left(1-\frac{1}{p}\right)}^{-k}
\end{equation*}
\begin{equation}
=\prod_{p \text{ prime }} \left(1-\frac{\nu_p(H \cup \{h_0\})}{\phi_1(p)}\right){\left(1-\frac{1}{p}\right)}^{-k}.
\label{gotic}
\end{equation}

When $p|D$ we get that $ \nu_p(H \cup \{h_0\})-1=0$, since $h_0 \equiv 0 \pmod D$, and $H$ was chosen such that $H \bmod D = \{0\}$.

When $p\nmid D$ it follows by a well-known result that $\phi_1(p)=p-1$. Hence, \eqref{gotic} becomes
\begin{equation*}
\mathfrak{S}(\rho')= \prod_{p|D} \left(1-\frac{1}{p}\right)^{-k} \prod_{p\nmid D} \left(1-\frac{\nu_p(H \cup \{h_0\})}{p-1}\right) \left(1-\frac{1}{p}\right)^{-k}.
\label{chestiein}
\end{equation*}
When $p|D$, 
\begin{equation*}
\left(1-\frac{1}{p}\right)^{-k}=\left(1-\frac{\nu_p(H \cup \{h_0\})}{p}\right){\left(1-\frac{1}{p}\right)}^{-(k+1)}.
\label{chestie2}
\end{equation*} 
When $p\nmid D$ we have that
\begin{equation*}
\left(1-\frac{\nu_p(H \cup \{h_0\})-1}{p-1}\right) \left(1-\frac{1}{p}\right)^{-k} = \displaystyle \frac{p-\nu_p(H \cup \{h_0\})}{p-1} \cdot \frac{p-1}{p}\left(1-\frac{1}{p}\right)^{-k-1}
\end{equation*}
\begin{equation*}
= \left(1-\frac{\nu_p(H \cup \{h_0\})}{p}\right){\left(1-\frac{1}{p}\right)}^{-(k+1)}.
\label{chestie1}
\end{equation*}
Combining the above two identities gives that $\mathfrak{S}(\rho')=\mathfrak{S}(H \cup \{h_0\})$. Hence the main term \eqref{mainterm3} is
\begin{equation}
 \sim \text{Li}(x) \frac{(\log R)^{k+2l}}{(k+2l)!} \binom{2l}{l} \mathfrak{S}(H \cup \{h_0\}).
 \end{equation}

Now we will prove that the error term \eqref{eroare2} is $o(x).$ Let 
$$E(x;q,a)=\left \lvert \pi^{*}(x;q,a)-\frac{\text{Li}(x)}{\phi_1(q)} \right \rvert.$$ 
Following a similar argument as in evaluating the error of the term \eqref{estimate 2}, we get that \eqref{eroare2} is
\begin{equation*}
\ll (\log R)^{2(k+l)} \displaystyle \sum_{d_1,d_2 \le R} \displaystyle \sum_{\substack{m= n \bmod \delta\\ \delta|P(m,H)}} E(x;\delta,m+h_0)
\end{equation*}
\begin{equation*}
 \ll (\log R)^{2(k+l)} \sum_{\delta \le R^2} (3k)^{\omega(\delta)} \max E(x; \delta, m+h_0).
 \end{equation*}
 The Cauchy-Schwarz inequality shows that the above is 
\begin{equation} 
\ll (\log R)^{2(k+l)} \left( \sum_{\delta \le R^2} \frac{(9k^2)^{\omega(\delta)}}{\delta}\right)^{1/2} {\left ( \sum_{\delta \le R^2} \delta  \max E(X; \delta, m+h_0)^2 \right)}^{1/2}.
\label{fog}
\end{equation}
Now, using \eqref{schwarz} we get that
\begin{equation*}
 \left( \sum_{\delta \le R^2} \frac{(9k^2)^{\omega(\delta)}}{\delta}\right )^{1/2} \ll \mathcal{L},
 \end{equation*} where we recall that $\mathcal L= O((\log N)^{O(1)}). $
 Using the asymptotics  
 \begin{equation*}
 \pi(x;q,a) \ll \frac{x}{q},
 \end{equation*} for $q \ll x,$ and \begin{equation*}
 \frac{\text{Li }(x)}{\phi_1(q)} \ll \frac{x}{q}
 \end{equation*} (which both follow by similar arguments as the ones used in evaluating the error of the term \eqref{estimate 2}), we get that 
 \begin{equation}
 E(x; \delta) \ll \frac{x}{\delta},
 \label{as4}
 \end{equation} where $E(x,q)=\displaystyle \max_{a}\displaystyle \max_{z \le x} E(z;q,a).$ 
 
 Now we define $\theta$ to be the supremum of all $\theta'$ for which $\displaystyle \sum_{q\le x^{\theta'}} E(x,q) \ll x/ (\log x)^A,$ for fixed $A$. Notice that by Theorem 6, it follows that $\theta \ge 1/2.$
 Hence, using the above asymptotics \eqref{as4}, we get that \eqref{fog} is equivalent to that the error term is
 \begin{equation}
  \ll \mathcal L \left(x \mathcal L \frac{x}{(\log x)^A}\right)^{1/2} \ll \frac{x}{(\log x)^{A'}},
  \end{equation} for $R^2 \ll x^{\theta-\epsilon}$. Thus, the error term is $o(x)$. 
  
 Since $\displaystyle \text{Li }(x) \sim \frac{x}{\log x}$, for $h_0 \notin H$ such that $h_0 \equiv 0 \pmod D$, for any $\epsilon>0$, we get the estimate
  \begin{equation}
\displaystyle \sum_{x \le n < 2x} \alpha(n) f(n+h_0)= \text{Li}(x) \frac{(\log R)^{k+2l}}{(k+2l)!} \binom{2l}{l} (\mathfrak{S}(H\cup \{h_0\})+o(\frac{1}{\log x})),
\label{est2}
\end{equation} provided $R \le x^{\theta/2-\epsilon}.$
  
  We use the same method to evaluate $\displaystyle \sum_{ x< n \le 2x} \alpha(n) f(n+h_0)$, for $h_0 \in H$ (by replacing the pair $(k,l)$ by $(k-1,l+1)$) and we get that for any $\epsilon>0$, we have the following asymptotics
 \begin{equation}
\displaystyle \sum_{x \le n<2x} \alpha(n) f(n+h_0)=  \text{Li }(x) \frac{(\log R)^{k+2l+1}}{(k+2l+1)!} \binom{2l+2}{l+1} (\mathfrak{S}(H)+o(\frac{1}{\log x})),
\label{est1}
\end{equation} provided $R \le x^{\theta/2-\epsilon}.$  

To get the desired result, we evaluate the ratio 
\begin{equation*}
\sum_{x \le n <2x} \alpha(n) \sum_{h \in H} f(n+h) / \sum_{x \le n <2x} \alpha(n).
\end{equation*}
Using the asymptotics \eqref{est1} and \eqref{orig}, it follows that for any $\delta > 0$, the inequality 
\begin{equation*}
\sum_{x \le n <2x} \alpha(n) \sum_{h \in H} f(n+h) / \sum_{x \le n <2x} \alpha(n) \ge \beta \frac{\text{Li }(x)}{x} \log R - \delta, 
\end{equation*} holds for sufficiently large $R$, where
\begin{equation*}
 \beta= \frac{2k(2l+1)}{(l+1)(k+2l+1)}.
 \end{equation*}
 Since $\text{Li }(x) \sim \displaystyle \frac{x}{\log x}$, we have  
 \begin{equation*}
\sum_{x \le n <2x} \alpha(n) \sum_{h \in H} f(n+h) / \sum_{x \le n <2x} \alpha(n) \ge \beta \frac{\log R}{\log x} - \delta.
\end{equation*} 
Notice that if we let $k, l \to \infty$ with $l=o(k)$, then $\beta \to 4$. If we assume an Elliott-Halberstam type conjecture for Fogels' generalization, then we can take $\theta>\frac{1}{2}$, and hence we get that
\begin{equation*}
\sum_{x \le n <2x} \alpha(n) \sum_{h \in H} f(n+h) / \sum_{x \le n <2x} \alpha(n) \ge 1-\delta.
\end{equation*} 
Letting $\delta \to 0$ proves that
\begin{equation*}
\sum_{x \le n <2x} \alpha(n) \sum_{h \in H} f(n+h) / \sum_{x \le n <2x} \alpha(n) \ge 1,
\end{equation*} for $R,x$ large enough. Letting $x \to \infty$ shows that for any $x$, there is some $n \in [x,2x)$ such that $n+h_i$ and $n+h_j$ split completely in the number field $K$, for some $i, j \in \{1, \ldots, k \}$. Hence, for any $x$, there are two primes that split completely between $x$ and $2x$, whose difference is bounded by $h_k-h_1$. This proves the existence of infinitely many primes that split completely, whose differences are bounded. 

The proof of the unconditional result \eqref{ec15} follows the same outline as the proof in section 2.4.
\end{proof}

\textbf{Acknowledgements}
\newline

The research for this paper has been done during the author's SURF (Summer Undergraduate Research Fellowship) at Caltech. The author would like to thank Dinakar Ramakrishnan and Paul Nelson for their guidance of the project and the Caltech SURF office for the funding.
\vspace{5mm}

\noindent Department of Mathematics, California Institute of Technology, Pasadena, California 91125
\newline
\textit{Email address:} \verb+amusat@caltech.edu+
\pagebreak

\bibliographystyle{plain}
\bibliography{ref3}

\end{document}